\documentclass[12pt,reqno,twoside]{amsart}

\usepackage{amsmath}

\usepackage{amssymb}

\usepackage[all]{xy}

\usepackage{epsfig}

\usepackage{color}

\numberwithin{equation}{section}

\font\sb = cmbx8 scaled \magstep0

\font\sn = cmssi8 scaled \magstep0

\long\def\combarak#1{\ifdraft{\sb #1 }\else\ignorespaces\fi}

\long\def\commargin#1{\ifdraft{\marginpar{\it #1}}\else\ignorespaces\fi}

\newcommand{\vt}{{\bf{t}}}
\newcommand{\vd}{{\bf{d}}}

\newcommand\mr{M_{m,n}}

\newcommand\vwma{very well multiplicatively  approximable}

\newcommand\vwa{very well approximable}
\newcommand\da{Diophantine approximation}
\newcommand\di{Diophantine}
\newcommand\de{Diophantine exponent}
\newcommand\hs{homogeneous space}

\newcommand\ssm{\smallsetminus}

\newcommand\lio{\Lambda \in \Omega}

\newif\ifdraft\drafttrue
\draftfalse

\newcommand\name[1]{\label{#1}{\ifdraft{\sn [#1]}\else\ignorespaces\fi}}
\newcommand\bname[1]{{\ifdraft{\sn [#1]}\else\ignorespaces\fi}}

\newcommand\eq[2]{{\ifdraft{\ \tt [#1]}\else\ignorespaces\fi}\begin{equation}\label{eq:#1}{#2}\end{equation}}

\newcommand {\equ}[1]     {\eqref{eq:#1}}

\newcommand\pa{$\varphi$-approximable}

\newcommand{\TT}{{\mathcal{T}}}

\newcommand{\compose}{{\circ}}

\newcommand{\R}{{\mathbb{R}}}

\newcommand{\Z}{{\mathbb{Z}}}

\newcommand{\N}{{\mathbb{N}}}

\newcommand{\cl}{\overline}

\newcommand{\vv}{{\bf{v}}}

\newcommand{\const}{{\operatorname{const}}}

\newcommand{\SL}{\operatorname{SL}}

\newcommand{\ggm}{G/\Gamma}

\newcommand{\Lie}{\operatorname{Lie}}

\newcommand{\diag}{{\rm diag}}

\newcommand {\ignore}[1]  {}

\newcommand{\interior}{{\rm int}}

\newcommand{\dist}{{\rm dist}}

\newcommand{\cL}{{\mathcal L}}

\newcommand{\cM}{{\mathcal M}}

\newcommand{\cA}{{\mathcal A}}

\newcommand{\cT}{{\mathcal T}}

\newcommand{\cR}{{\mathcal R}}

\newcommand{\df}{{\, \stackrel{\mathrm{def}}{=}\, }}

\newcommand{\cF}{{\mathcal{F}}}

\newcommand{\x}{{\mathbf{x}}}

\newcommand{\vr}{{\bf r}}

\newcommand{\p}{{\bf p}}

\newcommand{\til}{\widetilde}

\newcommand{\supp}{{\rm supp}}

\newcommand{\BA}{{\bold{Bad}}}

\newcommand{\sm}{\smallsetminus}

\newcommand{\vre}{\varepsilon}

\newcommand\cag{$(C,\alpha)$-good}

\newcommand{\vf}{{\bf{f}}}

\newcommand\nz{\smallsetminus \{0\}}

\newcommand{\fa}{{\mathcal A}}

\newcommand\dt{Dirichlet's Theorem}

\newcommand{\fr}{{\mathcal R}}

\newcommand{\vq}{{\mathbf{q}}}

\newcommand{\vp}{{\bf p}}

\newcommand{\va}{{\bf a}}

\newcommand\amr{$Y\in M_{m,n}%(\br)
$}

\newcommand\dc{Dirichlet constant}

\newcommand{\DI}{{\mathrm{DI}}}

\newtheorem{thm}{Theorem}[section]

\newtheorem{lem}[thm]{Lemma}

\newtheorem{prop}[thm]{Proposition}

\newtheorem{cor}[thm]{Corollary}

\newtheorem{dfn}[thm]{Definition}

\title[An `almost all versus no' dichotomy]{%The inheritance %question and Dirichlet improvability
An `almost all versus no' dichotomy in homogeneous dynamics and Diophantine approximation}

\author{Dmitry Kleinbock}

\address{Brandeis University, Waltham MA
02454-9110 {\tt kleinboc@brandeis.edu}}

\subjclass{11J13; 37A17}

\date{August 2009}

\begin{document}

     \begin{abstract}
    % It has been observed that several 
    Let $Y_0$ be a not \vwa\ $m\times n$ matrix, and let $\cM$ be a connected analytic
submanifold in the space of $m\times n$ matrices containing $Y_0$. Then almost all $Y\in\cM$
are not \vwa. This and other similar statements are cast in terms of properties of certain
orbits on \hs s and deduced from quantitative nondivergence estimates
for
`quasi-polynomial'
%correspondence between multidimensional \da\  and 
flows 
on  the space
of lattices.

\end{abstract}

\maketitle

\centerline{\it Dedicated to S.G.\ Dani on the occasion of his 60th birthday}

\section{Introduction}\name{intro}

This work is motivated by a result from a recent paper  \cite{dani survey} by S.G.\ Dani. 
Let $G$ be a connected Lie group and $\Gamma$ a lattice in $G$. Suppose 
%$\cA$ is an %sequence
%of elements in 
%unbounded subset of a maximal connected diagonalizable subgroup of $G$, 
$a$ is a semisimple element of $G$, and let
$$U = \{u\in G : a^{-n}ua^{n} \to e\text{ as }n\to\infty\}$$be the  {\sl expanding horospherical
subgroup\/} with respect to $a$. %the sequence $\{a_i\}$). 
%$\cA$, to be abbreviated by EHS). 
%It can be seen that $U$ as above is
%a simply connected nilpotent Lie group.  
%denote by $m$ a Haar measure on $U$. 
Now suppose that $U$ is not contained
in any proper closed normal subgroup of $G$,  
take an arbitrary sequence of natural numbers $n_k\to\infty$, and denote
by $\cA$ the set  $\{a^{n_k}  : k\in\N\}$. Then it follows from results of N.\ Shah
\cite{Shah Indian academy} that for any $x\in\ggm$, the set % and $m$-almost every 
\eq{conull}{\big\{u\in U :\cA ux  \text{
is dense in }\ggm\big\}} has full (Haar) measure.

One of the themes in \cite{dani survey} is a close investigation %of the structure
of  sets of type \equ{conull}. Namely, the following is a special case of \cite[Corollary 2.3]{dani survey}:

\begin{thm}
\name{thm: dani}
Let $G$, $\Gamma$, $a$, $\cA$ and $U$ be as above, and let $\{u_t: t\in\R\}$ be a one-parameter subgroup of $U$.  Suppose 
that for some $t_0\in\R$ and $x\in\ggm$,  $ \cA u_{t_0}x$ 
is dense in $\ggm$. Then  $ \cA u_tx$ 
is dense in $\ggm$ for almost all $t\in\R$. 
\end{thm}

In other words, an interesting dichotomy takes place: either a one-parameter subgroup of
$U$ is contained in the complement to the set \equ{conull}, or it intersects it in a set of full measure.

\smallskip

In this note we discuss other situations where  analogous conclusions can be derived. 
That is, we consider certain  properties %$\mathcal{P}$ 
of points in a big `ambient' set ($U$ in the above example)
 which happen to be generic (satisfied for almost all points in that set), and show that some
 `nice' measures $\mu$ on this set satisfy a similar dichotomy: that is, %$\mathcal{P}$
those properties  hold either
 for $\mu$-almost all points, or for  no points in the support of $\mu$.
In particular, such a phenomenon has been observed in metric theory of simultaneous \da,
which will be the main context in the present paper. 
 For  positive integers $m,n$, we denote by $\mr$ %stand for 
the space of $m\times n$ matrices with real entries; this will be the ambient space
of our interest. We will interpret elements \amr\  as systems
of $m$ linear forms in $n$ variables. Properties %$\mathcal{P}$ 
of $Y$ of our interest will be cast in terms of existence
or non-existence of not too large  integer vectors $\vq\in\Z^n$ such that  $\dist(Y\vq,\Z^m)$ 
%are close (or not too close) to integers $\vp\in\Z^m$. 
is small. Here are two %important 
examples.
%, which are special cases of
%more general situations we are going to consider.

\begin{dfn} \name{dfn de} The \de\  %\amr\ 
 $\omega(Y)$ of \amr\ %(sometimes called `the exact order' of $A$) 
 is the supremum of $v > 0$ for which 
 \eq{de}{
 %{\sl \vwa\/} (to be abbreviated by VWA) if for some $\vre > 0$ %such that such that 
%there exist infinitely many
%  $\vq\in\bz^n$ with
%\eq{vwa}{
%\exists \, \infty \text{ many }
%there exist infinitely many $\vq\in\Z^n$ such that 
\dist(Y\vq,\Z^m) < \|\vq\|^{ - v}\text{ for
infinitely many }\vq\in
\Z^n
\,.}%\text{ for some }\vp\in\Z^m
%\,.}
\end{dfn}
%rather small,
%that 
%It immediately follows from the Borel-Cantelli Lemma that 
%Lebesgue-almost every \amr\ is not VWA.

Here $\|\cdot\|$ and `$\dist$' depend on the choice of norms, but the above definition
does not.  
It is easy to see  that $\omega(Y) = n/m$ for  Lebesgue almost all \amr; those $Y$ for which 
$\omega(Y)$ is strictly bigger than $n/m$ are called {\sl very well approximable\/} (VWA).
%With some abuse of notation we will denote by $\bf VWA$ the set of  very well approximable \amr,
%hoping that dimensionality will be clear from the context.

\begin{dfn} \name{dfn sing} Let $\varphi:\R_+\to\R_+$ be a %bounded 
\commargin{`nonincreasing continuous' is not essential and can be replaced by `bounded', but 
to simplify things a bit, 
let's keep them for now}
non-increasing continuous function.
Say that  \amr\  is
{\sl $\varphi$-singular\/} if for any $c>0$ there is $N_0$ such that for all
$N \geq N_0$ one
can find %$\vp \in \Z^n$ and 
$\vq \in \Z^n\nz$ with
\eq{singular}{
\dist(Y\vq,\Z^m)  < \frac{c \varphi(N)}{N^{m/n}}  \text{ and }
\|\vq\|  < c N\,.}
\end{dfn}

As the previous one, this definition is norm-independent.
One says  that $Y$ is {\sl singular\/} if it is $\varphi$-singular with $\varphi\equiv 1$.
Note that the property of being $\varphi$-singular depends only on the equivalence class of $\varphi$ 
(its tail up to a multiplicative constant) and holds for  Lebesgue almost no $Y$ as long as $\varphi$ is bounded, as shown by Khintchine. 
%On the other hand, Khintchine (see also related work
%of Dani \cite{Dani-div} and Weiss \cite{cassels})  showed that
%if $\max(m,n) > 1$ and $\varphi$ is positive, then there exist uncountably many
%$\varphi$-singular
%matrices which are $\varphi$-singular `for non-obvious reasons'.  \commargin{how to say it better???}
Also, in view of Khintchine's Transference Principle, see \cite[Chapter V]{Cassels}, $Y$  is  singular or  very well approximable 
if and only if so is its transpose.  
%We will denote by $\varphi$-${\bf Sing}$ (resp., $\bf Sing$) the sets
%of $\varphi$-singular (resp., singular) \amr.

\smallskip

Proving that almost all $Y$ with respect
to some natural  measures other than Lebesgue do not have the above (and some other similar) properties
has been an active direction of research. Its motivation comes from a conjecture of Mahler \cite{mahler} 
(1932,
settled by Sprind\v zuk in 1964, see \cite{Sprindzuk-original, Sprindzuk-Mahler})
that $%Y = 
\begin{pmatrix}x & x^2 & \dots & x^n\end{pmatrix}\in M_{n,1}\cong \R^n$  is not VWA
for Lebesgue almost every $x\in\R$. In other words, in Sprind\v zuk's terminology, 
the curve %parameterized by
 \eq{mahler}{%{\mathcal V}_n \df 
 \{\begin{pmatrix}x & x^2 & \dots & x^n\end{pmatrix} \ :  \ x\in \R\}} is
{\sl extremal\/}. Later \cite[Theorem A]{KM}  the same conclusion was established for %analytic
submanifolds of $ \R^n$ of the form
 %\eq{spr}{%{\mathcal V}_n \df 
%
 \{%\begin{pmatrix}
$ \vf(\x) 
 %& f_2(\x) & \dots & f_n(\x)\end{pmatrix} \ 
 :  \x\in U\}$, where $U\subset \R^d$ is open and connected,
$\vf = (f_1,\dots,f_n): U\to \R^n$ is real analytic,  and
 \eq{cond on f}{%\begin{aligned}
% \vf = (f_1,\dots,f_n)\text{ is real analytic, and }\qquad\\
1, f_1,\dots,f_n\text{ are linearly independent over }\R\,.
% \end{aligned}
}
 %which are not contained in any affine hyperplane. 
 This settled a conjecture
made by Sprind\v zuk in 1980 \cite{Sprindzuk-Uspekhi}. A version of this result with `VWA' replaced by `singular' 
%a special case of one of the theorems
%from 
can be found in \cite{sing} and in a stronger form in \cite{di}.

%\end{itemize}

On the other hand, it is easy to construct examples of non-extremal analytic submanifolds of $\R^n$, %^or, more generally, $\mr$
%satisfying neither of the aforementioned conclusions 
see \S\ref{examples}
for more detail. More precisely, % the first named
%author 
a necessary and sufficient condition  for 
the extremality of an affine subspace $\cL\subset \R^n$ is given in \cite{dima gafa}.
% to 
%have the property that its almost every point is not VWA. 
This condition is 
explicitly written in terms of coefficients of  parameterizing maps for $\cL$, and, incidentally, 
it is shown that $\cL$ is not extremal  if and only if 
all its points are VWA.
%it a subset of $\bf VWA$.
Furthermore, the same dichotomy %has been shown to 
holds for any connected analytic
submanifold $\cM$ of $ \R^n$: either almost every\footnote{This will always mean `with respect to the smooth measure class on $\cM$'.}   point of $\cM$ is not VWA, or all points of $\cM$ are.
In \cite{dima gafa} this has been  done by finding an explicit necessary and sufficient condition involving 
the smallest affine subspace $\cL$ containing $\cM$ and via so-called 
 `inheritance theorems' generalizing the work in \cite{KM}.  See also  \cite{dima tams, yuqing} for extensions of these results with VWA replaced by `having \de\ bigger than $v$' for an arbitrary $v$.
 
 Note that the proofs in  \cite{KM} and subsequent  papers are based on homogeneous dynamics, that is,
 on quantitative nondivergence estimates for flows on the spaces of lattices, and on a possibility to phrase
 \de s and other characteristics in terms of the behavior of certain orbits. 

\smallskip
 
 In the present paper we give a simple argument showing how for arbitrary $m,n$ 
 the dichotomy
 described above can be directly (without writing explicit necessary
 and sufficient conditions) derived from quantitative nondivergence. Here is a special 
 case of our main result:

\begin{thm}
\name{thm: vwa} %For $m,n\in\N$, l
Let $\cM\subset \mr$ be a connected analytic
submanifold.
\begin{itemize}
\item[(a)] Let $v\ge n/m$ and suppose that  $\omega(Y_0) \le v$ for some  $Y_0\in\cM$; then $\omega(Y) \le v$ for almost every $Y\in\cM$.
\item[(b)] Let $\varphi:\N\to\R_+$ be as in Definition \ref{dfn sing},  and suppose that $\exists\,Y_0\in\cM$ which is not
$\varphi$-singular; then $Y$ is not
$\varphi$-singular for almost every  $Y\in\cM$.
\end{itemize}
\end{thm}

In other words, the aforementioned \di\ properties\footnote{Recall that 
$Y$ is called {\sl \di\/} if $\omega(Y) < \infty$. 
We remark that in \cite{chinese} a weaker result has %recently 
been obtained by elementary methods (not using estimates
on \hs s) in the case $m = 1$: if a connected analytic submanifold $\cM$ of $\R^n$ contains a
\di\ vector, then almost all vectors in $\cM$ are \di.} hold either for almost all or for no $Y\in\cM$. In particular, $\cM$ is extremal if and only if
it contains at least one not \vwa\ point.
\smallskip

Clearly (by Fubini's Theorem) the properties discussed in the above theorem
hold for almost every translate of an arbitrary $\cM$.
% even in the strongest possible
%cases, when $v = n/m$ and $\varphi$ bounded from below. 
It is also clear that if
$\cM$ belongs to a proper  rational affine subspace, that is, if 
$Y\vq \in\Z^m$ for some $\vq\in\Z^n\nz$ and all $Y\in\cM$, then 
all points of $\cM$ have infinite \de s and are $\varphi$-singular for arbitrary 
positive $\varphi$.
%We will refer to those example as `rational'.
\commargin{need better terminology!}
However %nontrivial (`irrational') examples also 
there exist less trivial examples of those exceptional subspaces; these will be discussed in 
\S\ref{examples}.  
%We will show in this paper that for any 
%$v < \infty$ (resp.,  for any positive function 
%$\varphi$) there exist
%`irrational' submanifolds, in fact proper affine subspaces, of  $\R^n$ viewed as
%the space of row matrices (linear forms), whose every point has
%\de\ bigger than $v$ (resp.\ is $\varphi$-singular).

\ignore{simple formula holds for \de\ $\omega(\cL)$
of affine subspaces $\cL$ of $ \R^n$ (that is, the real number $v$ such that
$\omega(Y) = v$ for almost all $Y\in \cL$). Namely, the following is true:

\begin{thm}
\name{thm: formula} Let an $s$-dimensional affine subspace $\cL$ of $ \R^n$ be parametrized by
\eq{L}{\x\mapsto(\x,
\x A' + \va_0)\,,}
where 
$A'\in M_{s,n-s}$ 
%is a matrix of size ${s} \times(n-{s})$ 
and $\va_0\in\R^{n-{s}}$ (here both $\x$ and $\va_0$ are
row vectors). Then $\omega(\cL) = \max\big(n,\omega(A)\big)$,
where $A = \begin{pmatrix}
\va_0 \\ A'
\end{pmatrix}\in M_{s+1,n-s}$.\end{thm}

This answers a question asked in \cite{dima tams}, where the above formula was proved
in several special cases.
Similarly (see \S \ref{appl}) one can write a rather simple necessary and sufficient condition
for a function $\varphi$ and a matrix $A$ equivalent to almost every
($\Leftrightarrow$ at least one) point
of $\cL$ parameterized as in \equ{L} being not $\varphi$-singular. }
\smallskip

As was the case in the papers  \cite{dima gafa, dima tams, yuqing}, Theorem \ref{thm: vwa} is deduced from  statements involving orbits on the space of lattices.
Namely, let us put \eq{notation}{k = m+n,  \ G = \SL_{k}(\R),  \ \Gamma = \SL_{k}(\Z)\text{
and }\Omega = \ggm\,.} The space $\Omega$  can be viewed as the space of unimodular lattices in $\R^k$ by means
of the correspondence $g\Gamma \mapsto g \Z^k$.
%$k=m+n$, and let 
Denote by $\fa
%_{m,n}
$
the set of %$k$-tuples 
$\vt = (t_1,\dots,t_{k})\in \R^{k}$
such that
$$t_1,\dots,t_{k} > 0%,\ t_{m+1},\dots,t_{k} < 0,
\quad \mathrm{and}\quad 
\sum_{i = 1}^m t_i =\sum_{j = 1}^{n} t_{m+j} \,.$$ For $\vt \in
\fa$  write  
\eq{defn gt}{g_{\vt} \df
\diag(e^{t_1}, \ldots,
e^{ t_m}, e^{-t_{m+1}}, \ldots,
e^{-t_{k}})\in G\,.
}
We will consider subsets %$\cA$ 
of $G$ of the form% \eq{def a}{%\cA = 
$$g_{{}_\cT} \df \{g_\vt : \vt\in \cT\}\,,
\text{ where $\mathcal{T} \subset \fa$ is  unbounded}\,,$$%}
%where 
 %$\mathcal{T} \subset \fa$ is  unbounded. 
 and will study their action on $\Omega$.
 Also for \amr\ define \commargin{I got rid of $\tau$ since there are way too many letters $T$ in various fonts}
%\eq{eq: defn tau}{
$$
u_Y \stackrel{\mathrm{def}}{=} \left(
\begin{array}{ccccc}
I_{m} & Y \\ 0 & I_{n} 
\end{array}
\right), %\ \ \ \bar{\tau} \stackrel{\mathrm{def}}{=} \pi \compose \tau\,,
$$
%}
where $I_{\ell}$ stands for the $\ell\times \ell$ identity matrix. %and let 
%$U = \tau(\mr)$.  
Then it is clear that the group $\{u_Y : Y\in\mr\}$ is the expanding horospherical subgroup of $G$
corresponding to $g_{\vt}$ where $\vt$ belongs to the `central ray' in $\fa$, that is, to
\eq{def r}{\fr%_{m,n}
 \df \left\{\left(\tfrac t m,\dots,\tfrac t m,\tfrac t n,\dots,
\tfrac t n\right)
: t > 0\right\}\,.}
%; in particular, it is contained in the EHS corresponding
%to any $\cA$ of the form \equ{def a}.

%Elements of the space $\Omega \df \ggm$ can be identified with unimodular lattices in $\R^k$ by means
%of the correspondence $g\Gamma \mapsto g \Z^k$.
Our goal is to show a dichotomy similar to (and in fact, generalizing) the one from Theorem \ref{thm: vwa} for certain properties of $ g_{{}_\cT}$-orbits on $\Omega$. Fix a norm $\|\cdot\|$ on $\R^k$ and define 
a function $\delta: \Omega \to \R_+$ by %saying that for $\Lambda\in\Omega$, 
%$\delta(\Lambda)$ is equal to the norm of a nonzero element of $\Gamma$ with the
%smallest norm, that is, 
$$
\delta(\Lambda) \df \inf_{\vv\in\Lambda\nz}\|\vv\|\ \ \text{ for } \Lambda\in\Omega\,.
$$
$\Omega$ is a noncompact space, and % the restriction of
the function
$\delta$ defined above %to this space 
can be used to describe
its geometry at infinity. Namely, Mahler's Compactness Criterion 
(see \cite{Rag} or \cite{Bekka-Mayer-book}) says that a subset of $\Omega$ is relatively
compact if and only if $\delta$ is bounded away from zero on this subset.
Further, it follows from the  reduction
theory for $\SL_{k}(\Z)$ %, see e.g.\  \cite[Satz 4]{siegel},
 that 
the ratio of $1 + \log\big(1/{\delta(\cdot)}\big)$ 
and $1 + \text{dist}(\cdot, \Z^{k})$ is bounded between two 
positive constants for any right invariant Riemannian metric 
`dist' on $\Omega$. In other words, a lattice $\lio$ for
which $\delta(\Lambda)$ is small is approximately 
$- \log{\delta(\Lambda)}$ away from the base point 
$\Z^{k}$. %The reader is referred to \cite{K1} for more details.
This justifies the following 

\begin{dfn} \name{dfn ge} For an unbounded subset 
 $\mathcal{T}$ of $\fa$ and $\lio$, % and  $\gamma \ge 0$, let us 
 %say that the trajectory 
%$ g_\cT \Lambda$  {\sl grows with exponent\/} 
%$> \gamma$ if for some $\vre > 0$ there exists an unbounded set of 
%$\vt\in \cT$ such that 
%$$
%\delta( g_{\vt} \Lambda) \le e^{-(\gamma + \vre) \|\vt\|}\,.%\tag 2.25 
%\text{\ for \ infinitely \ many \ }\T \in \Z_{\sssize +}^n\,.
%$$
% Also 
define the {\sl growth exponent\/}
$\gamma_{{}_\cT}(\Lambda)$ of $\Lambda$ with respect to $\cT$ by
%to be the supremum of all $\gamma > 0$ for which $ g_\cT \Lambda$ grows with exponent 
%$> \gamma$. In view of the preceding remark, one has
$$
\gamma_{{}_\cT}(\Lambda) \df \limsup_{\vt\to\infty,\,\vt\in\cT}
%\frac{\text{dist}(g_\vt\Lambda, \Z^{k})}{\|\vt\|}\,.
\frac{-\log\big(\delta(g_\vt\Lambda)\big)}{\|\vt\|}\,.
$$
\end{dfn}

In other words (in view of the  remark preceding the above definition), $
\gamma_{{}_\cT}(\Lambda) > \beta$ is equivalent to the existence of $\beta' > \beta$ 
such that $ \text{dist}(g_\vt\Lambda, \Z^{k}) \ge \beta'  \|\vt\|$ for an unbounded set of $\vt\in \cT$.
Even though the definition involve %choices of 
various norms, it  clearly does not depend on the choices of norms. Also the growth exponent
does not change if $\cT$ is replaced with another set of bounded Hausdorff distance form $\cT$,
so in what follows we can and will choose $\cT$ to be countable and with distance between its
different elements uniformly bounded from below. 
Note that it can be derived  from
the Borel-Cantelli lemma that for any
unbounded $\cT$, the  growth exponent\footnote{See also \cite{KM LL} for finer growth properties of almost all orbits on \hs s of Lie groups.} 
 of $\Lambda$ with respect to $\cT$ is equal to zero for Haar-almost all $\lio$. 

%However, in the present paper we are concerned with certain measures on $\Omega$
%other than Haar, and are interested in growth exponents of lattices generic with respect
%to those measures. Specifically, 

\smallskip

Here is another property related to asymptotic behavior of trajectories:

\begin{dfn} \name{dfn div} Given %an unbounded 
$\mathcal{T}\subset \fa$ and a bounded function
$\psi: \cT\to\R_+$, say that the trajectory $ g_{{}_\cT} \Lambda$ {\sl diverges faster than $\psi$\/}
\commargin{I know it does not sound appropriate, but let me keep this terminology for now} if 
$$
\limsup_{\vt\to\infty,\,\vt\in\cT}
%\frac{\text{dist}(g_\vt\Lambda, \Z^{k})}{\|\vt\|}\,.
\frac{\delta(g_\vt\Lambda)}{\psi(\vt)} = 0\,.
$$
In other words, if for every $c > 0$ one has ${\delta(g_\vt\Lambda)} < c\psi(\vt)$
for all $\vt\in \cT$ with large enough (depending on $c$) norm. 
\end{dfn}

Again this definition is insensitive to choices of norms, % and, assuming  bounded perturbations of the set $\cT$,
and also depends only on the behavior of $\psi$ at infinity up to a multiplicative constant.
An example: if $\psi\equiv 1$, the above condition, in view of
Mahler's Compactness Criterion, says that the trajectory $ g_{{}_\cT}  \Lambda$ diverges (that is, eventually leaves any compact subset of $\Omega$) as $\vt\to\infty$ in $\cT$.
Clearly, because of mixing of the $G$-action on $\Omega$, 
for any $\psi$ and $\cT$ as above, 
$ g_{{}_\cT} \Lambda$  diverges faster than $\varphi$ for Haar-almost no $\lio$.
\smallskip

In this paper we show:

\begin{thm}
\name{thm: ge} Suppose we are given $\lio$,  an unbounded $\mathcal{T}\subset \fa$, and 
a connected analytic
submanifold $\cM$ of $\mr$. Then:

\begin{itemize}
\item[(a)]  Let $\beta \ge 0$ and $Y_0\in\cM$  be such that $\gamma_{{}_\cT}(u_{Y_0}\Lambda) \le \beta$; then  $\gamma_{{}_\cT}(u_Y \Lambda) \le \beta$ for  almost all  $Y\in\cM$; 
%the growth exponent of  $u_Y \Lambda$ with respect to $\cT$ is constant
%for  almost every  $Y\in\cM$, and furthermore is equal to $\inf_{Y\in\cM}\gamma_{{}_\cT}(u_Y \Lambda)$;

\item[(b)] Let $\psi:\cT\to\R_+$ be  bounded,  and suppose that $\exists\,Y_0\in\cM$ such 
that the trajectory $ g_{{}_\cT} u_{Y_0} \Lambda$ does not  diverge faster than $\psi$; then $ g_{{}_\cT} u_{Y} \Lambda$ does not  diverge faster than $\psi$ for almost every  $Y\in\cM$.
\end{itemize}

 %and  $Y_0\in\cM$ 
%such that $\gamma_{\cT}\big(\tau(Y_0) \Lambda\big) \le c$.  
\end{thm}

%In other words, we establish   an `almost all versus no' 
%dichotomy similar to those from Theorems \ref{thm: dani} and  \ref{thm: vwa}.

\smallskip

A connection between the corresponding parts of   Theorems   \ref{thm: vwa} and \ref{thm: ge} is well known. Namely, it is observed by Dani \cite{Dani-div} that $Y$ is singular
if and only if the trajectory  $ g_{{}_\fr} u_Y \Z^k$ diverges, where $\fr$ is as in 
\equ{def r}, so that $ g_{{}_\fr}$ is a one-parameter semigroup.
Also  it follows from \cite[Theorem 8.5]{KM LL} that $Y$ is VWA iff the growth exponent of $u_Y \Z^k$ with respect to $\fr$ is positive, and moreover,
the latter growth exponent determines $\omega(Y)$. 
The aforementioned \di\ implications of Theorem \ref{thm: ge} 
correspond to the case $\cT= \fr$. However, choosing other unbounded subsets of $\fa$
also gives rise to interesting results, for example related to so-called multiplicative
approximation ($\cT = \fa$) or approximation with weights  ($\cT$ is a ray in $\fa$ different from $\fr$). We will comment on this in \S\ref{other}.

\smallskip

The structure of this paper is as follows. In the next section we prove Theorem \ref{thm: ge} using 
quantitative nondivergence estimates. Then in \S\ref{appl} we will go through the correspondence between
\da\ and dynamics,  and derive Theorem   \ref{thm: vwa} from Theorem \ref{thm: ge}. We also 
present other \di\ applications, including a solution to a matrix analogue of Mahler's Conjecture (Corollary \ref{cor: moscow}) suggested to the author by G.A.\ Margulis. In the last section we 
 bring up some conjectures and open questions, and also remark that the methods
 employed in this paper %, and hence the results,
are applicable to objects somewhat more general  than analytic submanifolds of $\mr$. %for example
%some fractal measures as described in \cite{friendly}, and state more general versions of 
%Theorems   \ref{thm: vwa} and \ref{thm: ge} following the methods of \cite{KMW}, 
%as well as.

\smallskip
{\bf Acknowledgements:}
This work has spurred out  of a joint project with Barak Weiss, in which we attempted %unsuccessfully 
%thus far, 
to answer some
of the still open questions asked in the last section of the paper.  Barak's  thoughtful remarks and insights are
gratefully acknowledged. Thanks are also due to Gregory Margulis, Elon Lindenstrauss,  Nimish Shah and the reviewer for useful comments.
The author was supported by %BSF grant 2000247, ISF grant 584/04 and 
NSF
Grant DMS-0801064.

%\vfil\eject

\section{Quantitative nondivergence and Theorem \ref{thm: ge}}\name{qn}

Notation: if $B = B(\x,r)$ is a ball in $\R^d$ and $c > 0$, $cB$ will denote the 
ball $B(\x,cr)$. Lebesgue measure on $\R^d$ will be denoted by $\lambda$.
%We need to recall the %definition of the 
%following property of real-valued functions
%on an open subset $U$ of $\R^d$: 
Given $C,\alpha > 0$ and $U\subset \R^d$, say that  a function $f:U\to \R$ is
         {\sl $(C,\alpha)$-good on $U$\/}
         if for any ball $B \subset U$ %centered in $\supp\,\nu$
and any
$\vre > 0$ one has
%%\eq{def-good}{
$$\lambda\big(\{\x\in B : |f(\x)| < \vre\}\big) \le C
\left(\frac{\varepsilon}{\sup_{\x\in B} |f(\x)|}\right)^\alpha{\lambda(B)}\,.$$%}
This property captures `quasi-polynomial' behavior of a function $f$.
See \cite{KM, friendly} for a discussion and many examples. The following 
proposition, which is essentially implied by \cite[Corollary 3.3]{dima gafa},  will be useful:

\begin{prop}
\name{prop: anal good}
Let $U$ be a connected open subset  of $\R^d$, and let $\cF$ be a finite-dimensional
space of  analytic real-valued functions on $U$. Then for any $\x\in U$ 
there exist $C,\alpha > 0$ 
\commargin{In fact I think $C$ and $\alpha$ can be chosen uniformly, but we probably
do not need it... or do we?} and a neighborhood $W\ni \x$ contained in $U$ such that
%such that %for any $\x_0\in K$ 
%and ball $W$ of radius not greater than $r$ and with $\overline {\sqrt{d}W}\subset U$ 
%there exist  $C,\alpha > 0$
%such that 
every $f\in\cF$ is \cag\ on $W$.
\end{prop}

\begin{proof} Without loss of generality we can assume that $\cF$ contains constant functions.
Let $1,f_1,\dots,f_N$ be the basis of $\cF$, and consider the map $\vf = (f_1,\dots,f_N):U\to\R^N$.
Then for any open subset $U'$ of $U$, $\vf(U')$ is not contained in any proper affine subspace of $\R^N$ (otherwise, in view of the analyticity of all the functions, the same would be true
for $\vf(U)$, hence the functions would not be linearly independent). Therefore, again due to analyticity, $\vf$ is nondegenerate
at every point of $U$ (see \cite{KM} for a definition)\commargin{or should we define 
nondegeneracy here?}, and the conclusion follows from  \cite[Proposition 3.4]{KM}.
%for any 
\end{proof}

Recall that given $m,n\in\N$ we fixed $k = m+n$ and defined
 $\Omega$ as in \equ{notation}.
 % = \ggm$, where $G = \SL_{k}(\R)$ and  $\Gamma = \SL_{k}(\Z)$,
%that is, the space of unimodular lattices in $\R^k$. 
In order to state the main measure estimate we need to introduce some more notation. 
Let
$$
\mathcal{W}\df\text{ the set of proper nonzero rational
subspaces of } \R^{k}\,.
$$
%Fix a Euclidean structure on $\R^{n+1}$. One can extend the Euclidean
%norm
%$\Vert \cdot \Vert$ from $\R^{n+1}$ to its exterior algebra.
From here until the end of this section we let $\Vert \cdot \Vert$ stand
for  the  Euclidean norm on $\R^{k}$, %induced by the standard inner product
%$\langle \cdot,\cdot \rangle$, 
which we extend from
$\R^{k}$ to its exterior algebra. %Then 
For $V
\in \mathcal{W}$ and $g \in G$, let
$$
\ell_V(g) \stackrel{\mathrm{def}}{=} \Vert g (\vv_1 \wedge \cdots \wedge
\vv_j ) \Vert\,,
$$
where $\{\vv_1, \ldots, \vv_j\}$ is a generating set for $\Z^{k} \cap
V$;
% and $\Vert \cdot \Vert$ is the extension of the Euclidean norm from
%$\R^{k}$ to its exterior algebra; 
note
that $\ell_V(g)$ does not depend on the choice of  $\{\vv_i\}$. 
%Also
%for any $\gamma \in\Gamma$ and $g\in G$ one has $\ell_V(g\gamma = \ell_V(g)$;
%that is, $\ell_V$ induces a function on $\Omega$ which, in order to cause confusion,
%will also be referred to as $\ell_V$. 

\smallskip

Let us record the following elementary observation:

\begin{lem}
\name{lem: minkowski} There exists a constant $E$ depending only on $k$
with the following property: for any  $V
\in \mathcal{W}$ and $g \in G$ there exists a one-dimensional rational subspace
$V'\subset V$ such that $\ell_{V'}(g)\le E \ell_V(g)^{1/\dim(V)}$. Consequently, one has
%\eq{delta}{
$$\delta(g\Z^k) \le E \cdot \inf_{V
\in \mathcal{W}} \ell_V(g)^{1/\dim(V)}\,.$$%}
\end{lem}

\begin{proof} Indeed, $\ell_V(g)$ by definition is the covolume of 
the lattice $g\Z^k\cap V$ in $V$, and Minkowski's Lemma, see \cite{Schmidt:book},
implies that such a lattice has a nonzero vector of length at most $\const\cdot \ell_V(g)^{1/\dim(V)}$
\commargin{Do we need more explanations?}
where the constant depends only on the dimension of $V$; thus one can choose
$V'$ to be the line passing through this vector.  \end{proof}

\smallskip
Here is the main measure estimate on which our argument is based:

\begin{thm}[\cite{dima tams}, Theorem 2.2]
\name{thm: friendly nondivergence}
Given $d,k\in\N$ and %an open $U \subset \R^n$,
positive constants $C,D,\alpha$,
         there exists $C_1 = C_1(d,k,C,\alpha) > 0$
with the following property.
Suppose  %a measure  $\nu$  on $\R^d$ is  $D$-Federer on %an open subset $U$ of $\R^d$,
 $\til B\subset \R^d$ %centered at  $\supp\,\nu$, 
is a ball, $0 < \rho \le 1$,
and $h$ is a continuous map $\til B
\to G$ 
%$\z\in U\cap\,\supp\,\nu$,
%and   $B%\subset U
%   %= B(\z,r)
%$
%is a ball centered at $\supp\,\nu$
such that
%the ball
%\footnote{Here and
%hereafter, for a ball $B=B(\x, r)\subset\R^n$ and a positive $c$, we
%denote $B(\x,cr)$
%%by $c
%      B$.}
%$3^nB$
%$\til{B} \df
%$B(\z,3^nr)\subset U$, and
%, $\til{B} = B(\z,3^nr)$, $\mathcal{H}$ is a
%family of continuous maps $X
%that 
for each $V \in \mathcal{W}$,
\begin{itemize}
\item[(i)]
the function
$ \ell_V\compose {h}$  is $(C,\alpha)$-good on $\til B%\subset U
%(\z,3^nr)
$,
%with respect to
%$\nu$,
           \end{itemize}
and
\begin{itemize} \item[(ii)]
\label{item: attain rho}
$ \ell_V \compose {h} (\x) \geq \rho^{\dim(V)}$ for some  $\x\in B = 3^{-(k-1)}\til B$.
          \end{itemize}
Then for any $\,0<
\varepsilon
\leq \rho$,
\eq{friendly nondiv}{
{\lambda\big(\big\{\x \in B: \delta\big({h}(\x)\Z^{k}\big)< \vre
\big\}\big)}\le C_1
\left(\frac{\vre}{\rho} \right)^{\alpha}{\lambda(B)} \,.
}
\end{thm}

This theorem is %very 
similar to its earlier versions, see \cite{KM, friendly}; however
one crucial difference is the term $\rho^{\dim(V)}$ in (ii), as opposed to just $\rho$
independent on the dimension of $V$. It is this improvement that will enable us to
prove sharp results.

\smallskip
Now recall that in the theorems stated in the introduction we are given a connected
analytic submanifold $\cM$ of $\mr$. We are going to parameterize it by an analytic map
$F:U\to\mr$, where $U$ is a connected open subset of $\R^d$, $d = \dim(\cM)$.
Theorem \ref{thm: friendly nondivergence} will be applied to 
$
h: U\to G$ given by \eq{def h}{h(\x) = g_\vt u_{F(\x)}g\,,} where $\vt\in\fa$ and $g\in G$ are fixed,
and our goal will be to check conditions (i) and (ii) of Theorem \ref{thm: friendly nondivergence}
and then use %the  estimate 
\equ{friendly nondiv}.

The next corollary (from Proposition \ref{prop: anal good}) will help us handle condition (i):

\begin{cor}
\name{cor: i}
Let $U$ be a connected open subset  of $\R^d$, and let $F:U\to\mr$ be an analytic map. 
Then for any $\x_0\in U$ 
there exist $C,\alpha > 0$ 
and a neighborhood $W\ni \x_0$ contained in $U$ such that
%such that %for any $\x_0\in K$ 
%and ball $W$ of radius not greater than $r$ and with $\overline {\sqrt{d}W}\subset U$ 
%there exist  $C,\alpha > 0$
%such that 
for any  $V \in \mathcal{W}$, $\vt\in\fa$ and $g\in G$, functions
$ \x\mapsto \ell_V \big(g_\vt u_{F(\x)}g\big)$ are $(C,\alpha)$-good on 
 $W$.
\end{cor}

\commargin{Not sure if we need this in the form of corollary instead of just mentioning it}

\begin{proof} For any $\vv_1,\dots,
\vv_j \in\R^k$, $\vt\in\fa$ and $g\in G$, 
the coordinates of $$g_\vt u_{F(\x)}g(\vv_1 \wedge \cdots \wedge
\vv_j)$$ in some fixed basis of the $j$-th exterior power of $\R^k$ are linear combinations
of products of matrix elements of $F$, where the number of factors in the products is 
uniformly bounded
from above. Therefore all those coordinate functions are analytic and span a finite-dimensional space,
and the claim follows from Proposition \ref{prop: anal good}
\commargin{or should we give more explanation about coordinates $\to$ norms?}
 and \cite[Lemma 4.1]{friendly}.\end{proof}

We are now ready for the

\begin{proof}[Proof of Theorem \ref{thm: ge}] 
Recall that we are given $\lio$ (which we will write in the form $g\Z^k$, where $g\in G$ is fixed),  an unbounded $\mathcal{T}\subset \fa$, and 
a connected analytic
submanifold $\cM$ of $\mr$ which we will parameterize by $F:U\to\mr$ where $U\subset \R^d$ is  open and connected. 

For part (a) we are given $\beta \ge 0$ such that the set
\eq{def a1}{A_1\df\{\x\in U : \gamma_{{}_\cT}(u_{F(\x)} \Lambda)  \le \beta\}}
is nonempty. Define
\eq{def a2}{A_2\df\{\x\in U : \lambda(B\ssm A_1) = 0 \text{ for some neighborhood $B$ of $\x$}\}\,.}
We claim that 
\eq{equality}{A_2 = \overline{A_1}\cap U\,.}
Since  $A_2$ is obviously open and $U$ is connected, this implies that $A_2 = U$, and therefore
$A_1$ has full measure, which is what we were supposed to show.

It is clear from \equ{def a2} that $A_2\subset \overline{A_1}$.  To prove equality in \equ{equality},
 take $\x_0\in \overline{A_1}$, and choose a ball $B\ni\x_0$ such that $\tilde B \df 3^{k-1}B$ is contained in $W$
as in Corollary \ref{cor: i}. This way, condition (i) of Theorem  \ref{thm: friendly nondivergence}
for $h$   as in \equ{def h} (uniformly in $\vt\in\cT$) is taken care of. Then choose $\x'\in B\cap A_1$; 
\equ{def a1} implies that
for any $\beta' > \beta$ and all large enough %(depending on $\beta'$)  
$\vt\in \cT$, one has $ \delta\big(g_\vt u_{F(\x_0)}g\Z^k\big) \ge e^{-\beta' \|\vt\|}$. 
Applying Lemma \ref{lem: minkowski}, we can conclude that $\ell_V \circ h(\x_0) 
%(g_\vt u_{F(\x_0)}g\big)
\ge (e^{-\beta' \|\vt\|}/E)^{\dim(V)}$ for any $V
\in \mathcal{W}$ and all large enough $\vt\in \cT$. % with  $\|\vt\| > t_0$. 
Thus condition (ii) of Theorem  \ref{thm: friendly nondivergence} is satisfied with $\rho = e^{-\beta' \|\vt\|}/E$. Taking $\vre = e^{-\beta'' \|\vt\|} $ where $\beta'' > \beta'$ is arbitrary,
we apply \equ{friendly nondiv} and conclude that for large enough $\vt\in \cT$,
\eq{bc estimate}{\lambda\big(\big\{\x \in B: \delta\big(g_\vt u_{F(\x)}\Lambda\big)< e^{-\beta'' \|\vt\|}
\big\}\big)\le C_1
E^\alpha e^{-\alpha(\beta'' - \beta') \|\vt\|} \lambda(B)\,.}
The sum of the right hand sides of the above inequality over all $\vt\in \cT$ is finite
(recall that $\cT$ is assumed to be `uniformly discrete'), hence   almost all $\x\in B$
belong to at most finitely many sets as in the left hand side of \equ{bc estimate}.
Since $\beta''$ can be arranged to be as close to $\beta$ as one wishes,
it follows that $\gamma_{{}_\cT}(u_{F(\x)} \Lambda) \le \beta$ for  almost all $\x\in B$,
that is $\x_0\in A_2$.
\smallskip

Part (b) is proved along the same lines: define 
%\eq{def a1 new}{
$$A_1\df\{\x\in U : g_{{}_\cT}(u_{F(\x)} \Lambda \text{ does not diverge faster than }\psi\}$$%}
and then $A_2$ by  \equ{def a2}; as before, the claim would follow from \equ{equality}.
Again,  take $\x'\in \overline{A_1}$ and $B\ni\x'$ such that $\tilde B \df 3^{k-1}B\subset W$
as in Corollary \ref{cor: i}, so that $h$ is as in \equ{def h} satisfies condition (i) of Theorem  \ref{thm: friendly nondivergence} for any $\vt$. Then  choose $\x_0\in B\cap A_1$. The latter implies that there exists
 $c > 0$ and an unbounded subset $\cT'$ of $\cT$ such that $$\delta(g_\vt u_{F(\x_0)}\Lambda) \ge  c\psi(\vt)\quad\forall\,\vt\in\cT'\,.$$ From Lemma \ref{lem: minkowski} it then follows that $$\ell_V \circ h(\x_0) 
%(g_\vt u_{F(\x_0)}g\big)
\ge \big(c\psi(\vt)/E\big)^{\dim(V)}$$ for any $V
\in \mathcal{W}$ and any $\vt\in \cT'$. Applying \equ{friendly nondiv}, we  conclude that for 
any $0 < \vre < 1$ and $\vt\in \cT'$, 
\eq{small measure estimate}{\lambda\big(\big\{\x \in B: \delta\big(g_\vt u_{F(\x)}\Lambda\big)<\vre c\psi(\vt)
\big\}\big)\le C_1
E^\alpha \vre^\alpha \lambda(B)\,.}
But by definition of `divergence faster than $\psi$' and since $\cT'$ is unbounded,
for any positive $\vre$ there exists
 $\vt\in \cT'$ such that  $B\ssm A_1$ is contained in the set 
in the left hand side of \equ{small measure estimate}. Hence $B\ssm A_1$ has measure zero,
which proves that $\x'\in A_2$. 
\end{proof}

%\pagebreak

%\vfil\eject

\section{\di\ applications}\name{appl}
\subsection{Proof of Theorem \ref{thm: vwa} }\name{pf vwa}
In order to connect Theorem \ref{thm: vwa} with  Theorem \ref{thm: ge}, we take $\cT = \cR$
as in \equ{def r}, and denote
%\eq{defn gt1}{
$$g_{t} \df
\diag(e^{t/m}, \ldots,
e^{ t/m}, e^{-t/n}, \ldots,
e^{-t/n})\,.$$
%}
According to \cite[Theorem 8.5]{KM LL}, \equ{de} holds if and only if
the inequality
$$
\delta(g_t u_{Y}\Z^k) < e^{-\frac{mv - n}{n(mv + 1)}t}
$$
is satisfied for an unbounded set of $t\in\R_+$. Consequently, one has $$\gamma_{{}_\cR}(u_Y \Z^k) = \frac{\frac mn\omega(Y) - 1}{ m\omega(Y) + 1 }\,,$$
and therefore Theorem \ref{thm: vwa}(a) follows immediately from Theorem \ref{thm: ge}(a).

The connection between parts (b) of these theorems is analogous. Given $\varphi$ as in 
Definition \ref{dfn sing}, define $N= N(t)$ by
\eq{def n}{
e^{ \frac{m+ n}{mn }t} = N^{1 + n/m} \varphi(N)^{-1}
}
\commargin{even just assuming boundedness it can be done}(this is well defined in view of the continuity and monotonicity of $\varphi$), 
and then let \eq{def psi}{\psi(t) = e^{-t/n} N \,.}
\commargin{do we need more explanations?}
Then, for any $c>0$,  $$e^{t/m}\frac{c \varphi(N)}{N^{m/n}} = e^{-t/n}c N = c\psi(t)\,;$$
thus the solvability
of \equ{singular} is equivalent to $\delta(g_t u_{Y}\Z^k) <  c\psi(t)$. 
Hence $Y$ is $\varphi$-singular if and only if $\gamma_{{}_\cR}(u_Y \Z^k)$ diverges faster than $\psi$
(here we identify $\cR$ with $\R_+$ and view $\psi$ as a function on $\cR$), which readily proves
Theorem \ref{thm: vwa}(b).
Note that given $\psi$ one can define $N$ by  \equ{def n} and then $\psi$ by  \equ{def psi},
thus there is at most one function $\varphi$ 
for which both \equ{def n} and \equ{def psi} hold. For example, $\varphi \equiv \const$
would give rise to $N(t) = e^{t/n}$ and thus $\psi(t)  \equiv \const$; 
and \commargin{should say it better}the faster is the decay of $\varphi$, the 
 faster would be the decay of $\psi$.
 
\subsection{Examples }\name{examples}  Here we take $m = 1$, that is, consider $\R^n$ as
the space of row vectors (linear forms). Let an $s$-dimensional affine subspace $\cL$ of $ \R^n$ be parametrized by
\eq{L}{\x\mapsto(\x,
\x A' + \va_0)\,,}
where 
$A'\in M_{s,n-s}$ 
%is a matrix of size ${s} \times(n-{s})$ 
and $\va_0\in\R^{n-{s}}$ (here both $\x$ and $\va_0$ are
row vectors). Denote by $\tilde \x$ the row vector $(1,\x)\in\R^{n+1}$, and put
$A = \begin{pmatrix}
\va_0 \\ A'
\end{pmatrix}\in M_{s+1,n-s}$. %Suppose that the \de\ $\omega(A)$ of $A$ is
%strictly bigger than $n$. Then 
It is easy to show, see  \cite[Lemma 5.4]{dima tams}, that all points of $\cL$ have \de s
at least as big as $\omega(A)$; in other words, a good rational approximation to $A$ gives rise to 
a good approximation to all points of $\cL$.  Choosing subspaces $\cL$ for which  $\omega(A)$
is arbitrary large one can produce examples of `irrational' subspaces consisting of arbitrarily well
approximable vectors. 
Similarly one can construct nontrivial examples of  subspaces consisting of $\varphi$-singular vectors.
For the sake of completeness let us work out those examples here, following the argument of 
 \cite[Lemma 5.4]{dima tams}.
Equation
\equ{L} can be rewritten as $\x\mapsto(\x,
\tilde \x A)$. Suppose $A$ is $\varphi$-singular (it is known from the work of Khintchine that
nontrivial examples of such matrices exist for any $\varphi$).
Then for any $c>0$ there is $N_0$ such that for all
$N \geq N_0$ one
can find $\vp = (p_0,p_1,\dots,p_s)\in\Z^{{s}+1}$ and 
$\vq \in \Z^{n-s}\nz$ such that
\eq{singular A}{
 \|A\vq + \vp\|  < \frac{c \varphi(N)}{N^{m/n}}  \text{ and }
\|\vq\|  < c N\,.}
Now take  any
$\x\in \R^{s}$, denote $(p_1,\dots,p_s)$ by $\vp'$  and write
$$
\left|\,p_0 + (\x,
\tilde\x A) \begin{pmatrix} \vp' \\ \vq
\end{pmatrix}\right|  = |p_0 + \x\vp' + \tilde\x A\vq| = |\tilde\x( A\vq +
\vp)|\le 
%({s}+1)
\left\|\tilde\x\right\|
\| A\vq + \vp\|\,.
$$ 
Therefore one has 
%\eq{good appr1}{
$\left|\,p_0 + (\x,
\tilde\x  A)\begin{pmatrix} \vp' \\ \vq
\end{pmatrix}\right| \le   \frac{C_1c \varphi(N)}{N^{m/n}}\,,
$
%\tag 5.4
%}
where $C_1$ depends only on $\x$.
Also, it follows from \equ{singular A}
that $\|\vp\|$ is bounded from above by $C_2|\vq\|$,
where $C_2$ depends only on $A$; hence
 $\left\|\begin{pmatrix} \vp' \\ \vq
\end{pmatrix}\right\| < C_2c N$. Since $c$ can be chosen to be arbitrary small,
it follows that $(\x,
\tilde \x A)$ is $\varphi$-singular for all $\x$.

\subsection{A matrix analogue of Mahler's Conjecture}\name{mahler}
We now describe
an  application of  Theorem  \ref{thm: vwa}  suggested to the author by G.A.\ Margulis. 
Given $m,n\in\N$, consider the $m^2$-dimensional submanifold of $M_{m,mn}$ given by 
 %\eq{mahlergeneral}{%{\mathcal V}_{m,n}\df
$$ \{\begin{pmatrix}X & X^2 & \dots & X^n\end{pmatrix} \ :  \ X\in M_{m,m}\},$$%}
which is a matrix analogue of  \equ{mahler}.
Then one can ask\footnote{This question was asked during the author's  talk at Moscow State University.} whether %${\mathcal V}_{m,n}$ 
the above manifold is extremal. 
% which by definition means that for any $\vre > 0$ and almost every $X$ there are at most finitely many $(\vq_1,\dots,\vq_n)\in(\Z^m)^n$ such that
%$$\dist(X\vq_1+\dots+X^n\vq_n,\Z^m) < \max(\|\vq_1\|,\dots,\|\vq_n\|)^{-(n+\vre)}\,.$$
%or, more generally, whether it is strongly extremal, i.e.\ for any $\vre > 0$ and almost every $X$ there %are at most finitely many $(\vq_1,\dots,\vq_n)\in(\Z^m)^n$ such that
%$$\Pi(X\vq_1+\dots+X^n\vq_n + \vp) < \prod_{i=1}^n \Pi_+(\vq_i)^{-(1+\vre)}$$
%for some $\vp\in\Z^m$. 
The answer turns out to be affirmative  and follows from the
 dichotomy established in Theorem  \ref{thm: vwa}. In fact a more general statement can be proved:

\begin{cor}
\name{cor: moscow}
Given $n\in\N$ and $v\ge n$, let  $\vf = (f_1,\dots,f_n)$ be an analytic map from a neighborhood
of $x_0\in\R$ to $%\R^n\cong 
M_{1,n}$, and suppose that $\omega\big(\vf(x_0)\big) \le v$.
Take $m\in \N$, and let $U$ be a neighborhood of $x_0I_m\in M_{m,m}$ such that 
the map \eq{def F}{F: X\mapsto\begin{pmatrix} f_1(X)  & \dots & f_n(X)\end{pmatrix}\in M_{m,mn}}
is defined for $X\in U$. Then  $\omega\big(F(X)\big) \le v$ for a.e.\ $X\in U$.
\end{cor}

In particular, if $\vf$ satisfies  \equ{cond on f}, then  $\omega\big(\vf(x)\big) = n$
for almost all $x$ in view of \cite[Theorem A]{KM}, hence the manifold $\{F(X): X\in U\}$ is extremal.

\begin{proof}[Proof of Corollary \ref{cor: moscow}] Note that $\omega\big(F(xI_m)\big) > v$ is equivalent to the existence of $w > v$ such that 
there are infinitely many 
$\vq = (\vq_1,\dots,\vq_n)\in(\Z^m)^n$ with
\eq{matrix de}{\dist\big(f_1(xI_m)\vq_1+\dots+f_n(xI_m)\vq_n,\Z^m\big) < %\max(\|\vq_1\|,\dots,\|\vq_n\|)
\|\vq\|^{-w}\,.}
(Here it is convenient to define $\|\cdot\|$ and `$\dist$' via the supremum norm.)
Write $\vq_i = (q_{i,1},\dots,q_{i,m})$ and $\vq^{(j)}= (q_{1,j},\dots,q_{n,j})$, and choose
$j = 1,\dots,m$ such that $ \|\vq\| = \|\vq^{(j)}\|$ for infinitely many $\vq$ satisfying \equ{matrix de}.
Then, since $ f_i(xI_m) = f_i(x)I_m$ for every $i$,  by looking at the $j$th component of vectors in the left hand side of \equ{matrix de} one concludes
that $\dist\big(f_1(x)q_{1,j}+\dots+f_n(x)q_{n,j},\Z\big) < \|\vq^{(j)}\|^{-w}$ for infinitely many $\vq^{(j)}\in \Z^n$, which implies $\omega\big(\vf(x)\big) > v$. Thus we have shown that $\omega\big(F(xI_m)\big) \le\omega\big(\vf(x)\big)$ whenever $\vf(x)$ is defined (the opposite inequality is also easy to show, although not needed for our purposes). The claim is therefore an immediate consequence
of Theorem  \ref{thm: vwa}(a). \end{proof}

Similarly one can conclude, using \cite{sing} and Theorem  \ref{thm: vwa}(b),  that under the assumption  \equ{cond on f} $F(X)$ as in \equ{def F} is not singular for a.e.\ $X$.

\subsection{Other applications }\name{other}  Here we describe two more corollaries from 
Theorem \ref{thm: ge} which deal with  \di\  properties more general than those
discussed in  Theorem \ref{thm: vwa}. 

\subsubsection{} For $\x = (x_i)\in\R^\ell$   define 
$$
\Pi(\x) \df \prod_{i = 1}^\ell |x_i|\quad \text{ and }\quad\Pi_{+}(\x) \df
\prod_{i = 1}^\ell \max(|x_i|, 1)\,.
$$
Then say that \amr\ is {\sl \vwma\/} (VWMA) if for some $\delta > 0$ there are infinitely many $\vq\in \Z^n$
such that %\eq{hom}{ 
$$\Pi(Y\vq  + \vp) < \Pi_{+}(\vq)^{-(1+\delta)}$$
%} 
for some $\vp\in\Z^m$.
Since $\Pi( Y\vq + \vp) \le \|Y\vq + \vp\|^m$ and
$\Pi_{+}(\vq)
\le
\|\vq\|^n$ for
$\vq\in\Z^n\nz$, VWA implies VWMA. Still it can be easily shown that  Lebesgue-a.e.\ $Y$ is not VWMA\footnote{Also it is known \cite{SW} that $Y$ is  VWMA
iff so is the transpose of $Y$.}. To show  that a submanifold $\cM$ of $\mr$ is {\sl strongly extremal\/}, that is,
its almost every point is not VVMA, is usually more difficult  than to prove its extremality. For example, 
the multiplicative version of Mahler's Conjecture, that is, the strong extremality of the curve 
%${\mathcal V}_{n}$ as in 
\equ{mahler} 
which was conjectured by Baker in the 1970s, has not been solved until the introduction of the methods of homogeneous dynamics to metric \da, and up to the present time there is no other proof than
the one from \cite{KM}. 
% (unlike its classical counterpart, which can be treated by extensions of
%Sprind\v zuk's methods developed by Bernik and Beresnevich). 
Note that applications of dynamics to multiplicative \di\
problems are based on the fact that $Y$ is VWMA if and only if 
$\gamma_{\fa}(u_Y\Z^{m+n} ) = 0$; that is, the orbit of the lattice $u_Y\Z^{k} $
under the action of the whole semigroup $\{g_\vt : \vt\in\fa\}$ has sublinear growth. 
This was shown in \cite{KM} and  \cite{friendly} in the cases $m = 1$ and $n = 1$ respectively.
The proof for the general case can be found in \cite{KMW}, see also \cite[Theorem 9.2]{KM LL} 
for a related statement. Therefore from Theorem \ref{thm: ge} one derives

\begin{cor}
\name{cor: strex}
A connected analytic submanifold  of $\mr$  is strongly extremal if and only if
it contains at least one not VWMA point.
\end{cor}

The case $\min(m,n) = 1$ of the above statement is established in \cite{dima gafa}. 

\subsubsection{} 
Let us generalize Definition \ref{dfn sing} as follows: suppose $\varphi:\fa_+\to\R_+$ is a %bounded 
function which is continuous and nonincreasing in each variable; that is, 
$$
\varphi(t_1,\dots,t_i,\dots,t_k) \ge \varphi(t_1,\dots,t_i',\dots,t_k) \quad\text{whenever} \quad t_i\le t_i' \,.$$
Also let $\cT$ be an unbounded subset of $\fa$. Now 
say that  \amr\  is
{\sl $(\varphi,\cT)$-singular\/} if for any $c>0$ there is $N_0$ such that for all $\vt\in\cT$ with 
$\|\vt\| \geq N_0$ one
can find %$\vp \in \Z^n$ and 
$\vq \in \Z^n\nz$ and $\vp\in\Z^m$ with
\[
\begin{cases}
|Y_i\vq - p_i| < c \varphi(\vt) e^{-t_i}\,,\quad &i = 1,\dots,m
 \\  
\ \ |q_j| < c \varphi(\vt)  e^{t_{m+j}}\,,\quad &j = 1,\dots,n
%\,.
\end{cases}
\]
In other words, those systems $Y$ of linear forms $Y_1,\dots,Y_m$ admit a drastic improvement
of the multiplicative (Minkowski's) form of Dirichlet's Theorem, see \cite{di} or \cite{nimish mult}.
It is not hard to show that the set of $(\varphi,\cT)$-singular matrices has Lebesgue measure zero
for any unbounded $\cT$. Arguing as in the proof of Theorem \ref{thm: vwa}(b), see \S \ref{pf vwa},
one can relate $(\varphi,\cT)$-singularity of $Y$ to  the trajectory $g_{{}_\cT'}u_Y\Z^{k} $ being 
divergent faster than $\psi$, where $\psi$ and $\cT'$ are determined by $ \varphi$ and $\cT$.
Thus from Theorem \ref{thm: ge} one can derive

\begin{cor}
\name{cor: t-sing} Let $\varphi$ and $\cT$ be as above, and suppose
a connected analytic submanifold $\cM$  of $\mr$ contains $Y_0$ which is not
$(\varphi,\cT)$-singular; then $Y$ is not
$(\varphi,\cT)$-singular for almost every  $Y\in\cM$.
\end{cor}

%Therefore one can ask for stronger sufficient conditions
%on a submanifold $\cM$ of $\mr$ 
%(for example of the form $F_*\lambda$
%submanifold $\Cal M$ of the space 
%where $F$ is a smooth map from an open subset of $\br^d$ to $\mr$)
%guaranteeing that it is {\sl strongly extremal\/}, that is,
%its almost every point is not VVWA.

\section{Generalizations and open questions}\name{general}

It seems natural to conjecture that other \di\ or dynamical properties might exhibit a dichotomy
of the same type as discussed in this paper. Here are some examples. For a 
function $\varphi:\N\to\R_+$  one says that \amr\ is 
{\sl \pa\/}  if  there are infinitely many
$\vq\in \Z^n$ such that 
$
 \|Y\vq + \vp\|   \le \varphi(\|\vq\|) $
for some
$\vp\in\Z^m$. 
(This definition is slightly different from the one
used in \cite{KM LL}, where powers of norms were considered.) 
The Khintchine-Groshev theorem gives the precise condition on the function
$\varphi$ under which the set of $\varphi$-approximable matrices has full
measure. 
 Namely, if $\varphi$ is non-increasing (this assumption can be removed in higher dimensions
but not for $n = 1$), then Lebesgue measure of the set of \pa\ 
\amr\
  is 
zero if \eq{kgt}
{\sum_{k = 1}^\infty {k^{n-1}\varphi(k)^m} < \infty\,,} 
and full otherwise. Now suppose \equ{kgt} holds and a connected analytic submanifold $\cM$ of $\mr$
contains a point which is not \pa; is it true that almost all $Y\in\cM$ are not \pa, or 
at least not $\tilde \varphi$-approximable, where $\tilde \varphi = C\varphi$
 with $C > 0$ depending on $Y$? Our methods are not powerful enough to answer this question. 
 Note that  \cite{KM LL}
provides a dynamical interpretation of $\varphi$-approximability along the lines of 
Definition \ref{dfn ge}. Namely, the choice of  $\varphi$ as above uniquely defines a continuous function $r:[t_0,\infty)\mapsto \R_+$ such that \amr\ is 
 \pa\ if and only if  there exist arbitrarily large positive  $t$
such that 
$$
\delta(g_{t}u_Y\Z^k)< r(t)\,.
$$ 

% One can also
%ask a similar question assuming that   \equ{kgt} does not hold and $\cM$ contains a \pa\ point.

\smallskip

Likewise, one can modify the definition of $\varphi$-singularity by fixing the constant $c$; 
as in the previous example, it is not clear if the `almost all versus no' dichotomy would still hold. Here is an important special case. 
Given positive $\vre < 1$, 
one says that 
Dirichlet's Theorem {\sl can be $\vre$-improved\/} for $Y$, 
%is {\sl $\vre$-Dirichlet improvable\/}, 
writing   $Y\in\DI_\vre$, 
if %there exists a positive $\vre < 1$ such that 
for every sufficiently
large $t$  one can find
%a nonzero integer 
$\vq  \in \Z^n\nz$ and 
$\vp \in \Z^m$ with
$$
\|Y\vq - \vp\|
%^m
 < \vre e^{-t/m} %{T^{1/m}}
  \ \ \ \mathrm{and}  
\ \ \|\vq\|
%^n 
 < \vre  e^{t/n} %T^{1/n}
\,.$$
Clearly $Y$ is  singular iff it belongs to $\cup_{\vre > 0}\DI_\vre$. It was proved
by Davenport and Schmidt \cite{Davenport-Schmidt2} that the sets $ \DI_\vre$
have Lebesgue measure zero. In fact, the latter statement follows from the ergodicity
of the $G$-action on $\ggm$: arguing as in \S \ref{pf vwa}, it is not hard to see
that $Y\in\DI_\vre$ iff the $g_{{}_\cR}$-orbit of $u_Y \Z^k$ misses a certain nonempty 
open subset of $\ggm$. This motivates questions extending both Theorem \ref{thm: dani} and  (in some direction) Theorem \ref{thm: ge}(b). Namely, 
let $G$, $\Gamma$, $a$, $\cA$ and $\{u_t: t\in\R\}$ be as in Theorem \ref{thm: dani}, and  suppose 
that for some $t_0\in\R$ and $x\in\ggm$, the trajectory
 $ \cA u_{t_0}x$ has a limit point in an open subset $W$ of $\ggm$. Is it true that
the intersection of  $\overline{ \cA u_{t }x} $ with $ W $ is nonempty  for almost all $t\in\R$? 
Or else let $G$ and $\Gamma$ be as in
 \equ{notation}, 
take an open subset  $W$ of $\ggm$ and   $\cT\subset \fa_+$,  and suppose that 
a connected analytic submanifold $\cM$ of $\mr$
contains a point $Y_0$ such that $g_{{}_\vt}u_{Y_0} \Z^k\in W$ for an unbounded set of $\vt\in\cT$;
then is the same  true for almost every $Y\in\cM$? An affirmative answer to the latter question would imply 
that for any positive $\vre < 1$ and any $\cM$ as above,  the set $\cM\ssm\DI_\vre $ is either
empty or of full measure. Note that it follows from the methods of proof of \cite{nimish} that
almost all $Y\in\cM$ are not in $\DI_\vre $ for any $\vre < 1$ whenever $\cM$ contains a point
$Y_0$ such that  the $g_{{}_\cR}$-orbit of $u_{Y_0} \Z^k$ is dense in $\ggm$.

\smallskip

Finally we would like to mention that the assumption of analyticity of manifolds $\cM$ in the main
results of the paper cannot be replaced by differentiability.  Indeed, it is not
hard to smoothly glue an extremal  $C^\infty$ curve in $\R^n$ to a rational line. 
On the other hand, one of  important advantages of the use of the quantitative nondivergence
method % to metric \da\ is 
has been a possibility to treat measures on $\mr$  other
than volume measures on analytic submanifolds. %This was first understood in \cite{friendly}
%and then further exploited in \cite{bad, sing, di, dima tams}. 
The reader is referred to  \cite{friendly, dima tams, di} and a recent paper
\cite{KMW} for a description of more general classes of measures  allowing a similar `almost all vs.\ no' dichotomy.

\end{document}